\documentclass[10pt,leqno]{amsart}

\usepackage{amsmath}
\usepackage{amsfonts}
\usepackage{amssymb}
\usepackage{graphicx}
\usepackage{color}
\usepackage{hyperref}
\usepackage{xcolor}

\newtheorem{theorem}{Theorem}[section]

\newtheorem*{theorem*}{Theorem}
\newtheorem{lemma}{Lemma}[section]
\newtheorem{corollary}[theorem]{Corollary}

\newtheorem{definition}[theorem]{Definition}

\newtheorem{remark}[theorem]{Remark}

\def\Ric{\text{Ric}}

\def\th{\theta}

\def\p{\partial}

\def\R{\mathbb{R}}

\def\k{\kappa}

\def\Ric{\operatorname{Ric}}

\numberwithin{equation}{section}

\begin{document}

\title[Parabolic equations on manifolds]{Schwarz symmetrizations in parabolic equations on complete manifolds}

\author{Haiqing Cheng}
\address{School of Mathematical Sciences, Soochow University, Suzhou, 215006, China}
\email{chq4523@163.com}

\author{Tengfei Ma}
\address{School of Mathematical Sciences, Soochow University, Suzhou, 215006, China}
\email{1829401184@stu.suda.edu.cn}

\author{Kui Wang} \thanks{The research of the third author is supported by NSFC No.11601359}
\address{School of Mathematical Sciences, Soochow University, Suzhou, 215006, China}
\email{kuiwang@suda.edu.cn}



\begin{abstract}
In this article, we prove a sharp  estimate for the solutions to parabolic equations on manifolds.  Precisely, using symmetrization techniques and isoperimetric inequalities on Riemannian manifold, we obtain a Bandle's comparison  on complete noncompact manifolds with nonnegative Ricci curvature and compact manifolds with positive Ricci curvature respectively.  Our results generalize Bandle's result \cite{Ba76}  to Riemannian setting, and Talenti's comparison for elliptic equation on manifolds by  Colladay-Langford-McDonald \cite{CLM18} and  Chen-Li \cite{CL21} to parabolic equations.
\end{abstract}

\maketitle

\section{Introduction}
Let $\Omega\subset\R^n$ be a bounded domain with smooth boundary, $\Omega^\sharp\subset\R^n$ be a round ball  with the same volume as $\Omega$, $f(x)$ and $g(x)$ be  nonnegative functions on $\Omega$, and $f^\sharp$, $g^\sharp$ be the Schwarz rearrangement of $f$ and $g$, see Definition \ref{sch} below. Let $u(x)$ and $v(x)$ be solutions to
\begin{align*}
    \begin{cases}
    u_t(x,t)-\Delta u(x,t)=f(x), & (x,t)\in \Omega\times(0,\infty),\\
    u(x, 0)= g(x), & x\in \Omega,\\
    u(x,t)=0, & (x,t)\in \p \Omega \times (0,\infty),
    \end{cases}
\end{align*}
and
\begin{align*}
    \begin{cases}
    v_t(x,t)-\Delta v(x,t)=f^\sharp(x), & (x,t)\in \Omega^\sharp\times(0,\infty),\\
    v(x, 0)= g^\sharp(x), & x\in \Omega^\sharp,\\
    v(x,t)=0, & (x,t)\in \p \Omega^\sharp \times (0,\infty).
    \end{cases}
\end{align*}
Then  for all $(a,t) \in(0,|\Omega|)\times[0, +\infty)$, it holds
\begin{align}\label{1-1}
  \int_0^a u^*(s,t)\, ds\le \int_0^a v^*(s,t)\,ds,
\end{align}
where $u^*$ is the decreasing rearrangement of $u$, see \eqref{def-dr} below. Inequality \eqref{1-1} was proved by Bandle \cite{Ba76} for strong solutions,  by V\'azquez \cite{Va82} and Mossino-Rakotoson \cite{MR86} for weak solutions. The elliptic version of \eqref{1-1} is known as Talenti's comparison \cite{Ta76, Ta79}. The proof of \eqref{1-1} is mainly based on the isoperimetric inequality and symmetrization techniques. Talent and Bandle's comparisons play an important role in mathematical studies, since there provide sharp estimates for solutions to elliptic and parabolic equations,  Faber-Krahn type inequality for  eigenvalues \cite{Ta76,ANT21}, and the bounds on the exit times of Brownian motion\cite{BS01,CLM18}.
Bandle's comparison \eqref{1-1} was generalized to nonlinear parabolic equations with Dirichlet boundary condition (see for instance \cite{ATL90,AVV10} and references therein),  and to parabolic equations with Neumann boundary condition  \cite{FM05} as well. We also refer the reader to the excellent books \cite{Ka85, Ke06} for related topics.

Recently, Talenti's comparison for Dirichlet boundary was generalized to the compact manifolds with positive Ricci curvature by Colladay, Langford and McDonald \cite{CLM18}, and to complete noncompact manifold with nonnegative Ricci curvature and positive asymptotic volume ratio by  Chen and Li \cite{CL21}. Also Talenti's comparison for Robin boundary was proved in both Euclidean space \cite{ANT21, ACNT21} and complete manifolds \cite{CLY21,CMW21}. To the best of our knowledge, Bandle's comparison, however, did not  receive  as much attention as Talenti's comparison on manifolds. The purpose of the present paper is to study the  Bandle's comparison for solutions to parabolic equation on manifolds.  In particular, we will establish a  sharp comparison between the solutions to \eqref{1-3} and \eqref{1-4} on  complete and noncompact manifolds with nonnegative Ricci curvature and positive asymptotic volume ratio, and compact manifolds with positive Ricci curvature, see Theorem \ref{thm} below.

Let $(M, g)$ be a  complete Riemannian manifold of dimension $n$, which is either compact with $\Ric\ge (n-1)\k$ for $\k>0$ or noncompact with nonnegative Ricci curvature and positive positive asymptotic volume ratio. Denote by
\begin{align}\label{1-2}
    \theta:=
    \begin{cases}
    \lim_{r\to \infty}\frac{|B_p(r)|}{\omega_n r^n}, & \quad \k=0,\\
    \frac{|M|}{|M_\k|}, &\quad \k>0,
    \end{cases}
\end{align}
where  $B_p(r)$ is the geodesic ball centered at $p$ with radius $r$ in $M$, $M_\k$ is the $n$ dimensional space form of constant sectional curvature $\k$, $\omega_n$ is the volume of the unit ball in $\R^n$, and $|M|$ denotes the volume of $M$. It follows from   the volume comparison that $\th\le 1$.
Let $\Omega\subset M$ be a bounded domain with smooth boundary, $f(x)$ and $g(x)$ be smooth and nonnegative function not identically zero on $\Omega$, we consider the following Cauchy problem
\begin{align}\label{1-3}
    \begin{cases}
    u_t(x,t)-\Delta u(x,t)= f(x), &\quad (x,t)\in \Omega\times(0,\infty) ,\\
    u(x,t)= 0, & \quad (x,t)\in \p \Omega\times(0,\infty),\\
	u(x,0)=g(x),& \quad x\in \Omega.
    \end{cases}
\end{align}
 Let  $\Omega^{\sharp}$ be a geodesic ball in $M_\k$ satisfying $\theta |\Omega^\sharp|=|\Omega|$,
 $f^\sharp$ and $g^\sharp$ defined on $\Omega^\sharp$ are the Schwarz rearrangement of $f$ and $g$ respectively. Denote by  $v(x)$  the solution to the following Schwarz rearrangement system of \eqref{1-3}
 \begin{align}\label{1-4}
    \begin{cases}
    v_t (x,t)-\Delta v(x,t)= f^\sharp(x), &\quad (x,t)\in \Omega^\sharp\times(0,\infty),\\
    v(x,t)= 0, & \quad (x,t)\in \p \Omega^\sharp\times(0,\infty),\\
	v(x,0)=g^\sharp(x),& \quad x\in \Omega^\sharp.
    \end{cases}
\end{align}
The main result of this paper is the following theorem, which gives a sharp Bandle's comparison on Riemannian manifolds.
\begin{theorem}\label{thm}
 Let $u(x)$ and $v(x)$ be the solutions to \eqref{1-3} and \eqref{1-4} respectively. Then
\begin{align}\label{1-5}
U(a, t)\le V(a, t),\qquad (a,t) \in (0,|\Omega^\sharp|]\times (0,\infty).
\end{align}
Where
\begin{align}\label{1-6}
  U(a,t)=\frac{1}{\theta }\int_{0}^{\theta a}u^*(s,t)ds,\quad V(a,t)= \int_{0}^{a}v^*(s,t)ds.
\end{align}
Moreover if inequality \eqref{1-5} holds as an equality for all   $(s,t)\in (0,|\Omega^\sharp|]\times (0,\infty)$, then $M$ is isometric to  $M_\k$,
$\Omega$ is isometric to  $\Omega^\sharp$, and $u(x,t)=v(x,t)$ on $\Omega^\sharp\times(0,\infty)$.
\end{theorem}
  \begin{remark}
In this paper, we mainly focus on Bandle's comparison \eqref{1-1} on manifolds and we then set up Theorem \ref{thm} in  smooth case. In fact, Theorem \ref{thm}  remains valid for weak solutions to \eqref{1-3}, see \cite{MR86} and \cite{ATL90}.
\end{remark}
The following corollary is a simple  consequence of \eqref{1-5}, which gives a sharp $L^p$ norm  estimate of $u$ by $v$  on space variables for all $t$.
\begin{corollary}\label{Cor}
 Let $u(x)$ and $v(x)$ be the solutions to \eqref{1-3} and \eqref{1-4} respectively. Then  we have
 \begin{align}\label{1-7}
    \Big (\frac 1 \th \int_\Omega u^p(x,t) \,dx\Big)^{1/p}\le \Big (\int_{\Omega^\sharp} v^p(x,t) \,dx\Big)^{1/p}
 \end{align}
for all $t>0$ and $p\in [1,\infty]$.
\end{corollary}
We mention finally that Talenti's comparison for Poisson  equation with Robin boundary has  recently been confirmed by  Alvino, Nitsch, and Trombetti in dimension two, as well as an integral estimate for higher dimensions \cite{ANT21}. It is a natural question to extend Bandle's comparison to heat equation with Robin boundary, and  we leave this to a future study.

This  paper  is  organized  as  follows.   In  Section  2 we  recall the Schwarz rearrangements and isoperimetric inequalities on manifolds.  In Section 3 we prove Theorem 1.1.

\section{Preliminaries}
 Let $\Omega$ be a bounded smooth domain in $M$ and $\Omega^\sharp$ be a geodesic ball in $M_\k$ (i.e. $\R^n$ or $\mathbb {S}^n$) with volume
$|\Omega|/\th$, where $\th$ is a constant defined by \eqref{1-2}.
We recall the definitions and properties of the Schwarz rearrangement of nonnegative functions on manifolds, see also  \cite[Section 2]{CL21} and \cite[Section 2]{CLM18}.
\begin{definition}\label{sch}
Let $h(x)$ be a nonnegative measurable function on $\Omega$. Denote by
$\Omega_{h,s}=\{x\in \Omega: h(x)>s\}$
and
$
\mu_h(s)=|\Omega_{h,s}|
$,
the decreasing rearrangement $h^*$ of $h$ is defined by
\begin{align}\label{def-dr}
    h^*(s)=\begin{cases}
    \operatorname{ess} \sup_{x\in \Omega} \, h(x), & \quad s=0,\\
    \inf\{t\ge 0: \mu_h(t)<s\}, & \quad s>0,
    \end{cases}
\end{align}
for $s\in [0, |\Omega|]$.  The Schwarz rearrangement of $h$ is defined by
\begin{align}\label{def-sr}
    h^\sharp(x)=h^*(\th \omega_n r^n(x)), \quad x\in \Omega^\sharp,
\end{align}
where $r(x)$ is the distance function from the center of $\Omega^\sharp$ in $M_\k$, and $\omega_n$ is the volume of unit ball in $\R^n$.
\end{definition}
It follows directly from \eqref{def-dr} and \eqref{def-sr} that
\begin{align}
    \mu_h(s)=\th \mu_{h^\sharp}(s)
\end{align}
for $s\ge 0$. Meanwhile, the Fubini's theorem gives
\begin{align}
    \int_\Omega h^p(x)\, dx=\int_0^{|\Omega|} (h^*)^p(s)\, ds=\th \int_{\Omega^\sharp} (h^\sharp)^p(x)\, dx,
\end{align}
for $h\in L^p(\Omega)$, $p\ge 1$. Moreover for any nonnegative functions $f(x)$ and $g(x)$,  the following inequality, known as  Hardy-Littlewood inequality, holds true
\begin{align}\label{2-4}
    \int_\Omega f(x) g(x)\,dx\le \int_0^{|\Omega|} f^*(s) g^*(s) \, ds.
\end{align}
Taking $g(x)$  as the characteristic function of $\Omega_{h,s}$ in above inequality yields
\begin{align}\label{2-6}
\int_{\Omega_{h,s}} f(x)\,dx\le \int_0^{\mu_h(s)} f^*(\eta)\,d\eta.
\end{align}

To prove Theorem \ref{thm}, we require the following isoperimetric inequality on manifolds with Ricci curvature bounded from below.
 \begin{theorem}
With $M$, $\th$, $\Omega$ and $\Omega^\sharp$ as above, there holds
\begin{align}\label{iso}
    |\p \Omega|\ge \theta  |\p \Omega^\sharp|,
\end{align}
where $|\p \Omega|$ denotes the $(n-1)$-dimensional area of $\p \Omega$. The equality holds if only if $\Omega$ is isometric to $\Omega^\sharp$.
 \end{theorem}
When $\Ric\ge n-1$, inequality \eqref{iso} was shown by L\'evy and Gromov \cite{Gr99},
see also Theorem 2.1 of \cite{NW16}. When $M$ is noncompact and $\Ric\ge 0$,  inequality \eqref{iso} was shown in dimension three by Agostiniani, Fogagnolo and Mazzieri \cite{AFM20} and in all dimensions by Brendle \cite{Bre21}. Meanwhile, Brendle proved in \cite{Bre21} that \eqref{iso} also holds true when $\Omega$ is a compact minimal submanifold of $M$ of dimension $n+2$ with nonnegative sectional curvature.

\section{Proof of Theorem \ref{thm}}
In this section, we will prove the main theorem.
For simplicity, we write $\Omega_{u}(a,t)$ and $\Omega_{u}  ^{\sharp } (a,t)$ as $\Omega_{a,t}$ and $\Omega^\sharp_{a,t}$ respectively for short.
Let
\begin{align*}
  \Phi(s)=|\p B_s|
 \end{align*}
for $s\in (0,|\Omega^\sharp|]$, where $B_s$ is a round geodesic ball in $M_\k$ with volume $s$. It can be easily checked that $\Phi(s)=n\omega_{n}^{1/n}s^{(n-1)/n}$ if $\k=0$, and $\Phi(s)$ is monotone increasing in $s$.

\begin{lemma}
Under the hypotheses of  Theorem \ref{thm}, we have
\begin{align}\label{3-1}
 \th^2\Phi^2(\frac{\mu_u(s,t)}{\th})\le -\p_s \mu_u(s,t)
\int_{ \Omega_{s,t}}f(x)-u_t(x,t) \, dx,
  \end{align}
and
\begin{align}\label{3-2}
\Phi^2(\mu_v(s,t))= -\p_s \mu_v(s,t)
\int_{ \Omega^\sharp_{s,t}}f^\sharp(x)-v_t(x,t) \, dx,
\end{align}
 for a.e. $t>0$.
\end{lemma}

\begin{proof}
By the Sard's theorem, for each $t>0$ we have
\begin{align}\label{1-3.3}
    \p \Omega_{s,t}=\{x\in \Omega: u(x,t)=s\}.
\end{align}
for almost every $s\ge 0$.
Observing from isoperimetric inequality \eqref{iso} that
\begin{align}\label{1-3.4}
   \theta \Phi (\frac{|\Omega_{s,t}|}{\th})\le&|\p \Omega_{s,t}|, \end{align}
and from the H\"older inequality that
 \begin{align}\label{1-3.5}
|\p \Omega_{s,t}|^2\le \int_{\p \Omega_{s,t}}\frac{1}{|\nabla u|} \, dA
\int_{\p \Omega_{s,t}}|\nabla u| \, dA,
\end{align}
then we get
\begin{align}\label{1-3.6}
  \theta^2\Phi^2(\frac{|\Omega_{s,t}|}{\th})\le \int_{\p \Omega_{s,t}}\frac{1}{|\nabla u|} \, dA
\int_{\p \Omega_{s,t}}|\nabla u| \, dA,
\end{align}
here and thereafter $dA$ denotes the induced measure on $(n-1)$ dimensional surface in $\Omega$.
Because $u(x,t)$ is vanishing on $\p \Omega$, so $\Omega_{s,t}$ is a compact set in $\Omega$, therefore by the coarea formula we have
\begin{align*}
    \p_s \mu_u(s,t)=-\int_{\p \Omega_{s,t}}\frac{1}{|\nabla u|} \, dA
\end{align*}
then \eqref{1-3.6} becomes to
\begin{align}\label{1-3.7}
  \theta^2\Phi^2(\frac{|\Omega_{s,t}|}{\th})\le \p_s \mu_u(s,t)
\int_{\p \Omega_{s,t}}|\nabla u| \, dA.  \end{align}
For $t>0$, noticing from Definition 2.1 and equality \eqref{1-3.3} that
$$
|\nabla u|=-\frac{\p u}{\p \nu}
$$
on $\p \Omega_{s,t}$  for a.e. $s>0$, where $\nu$ is the unit outer normal to $\p \Omega_{s,t}$, we then get
\begin{align*}
\int_{\p \Omega_{s,t}}|\nabla u| \, dA=\int_{ \Omega_{s,t}}-\Delta u(x) \, dx=\int_{ \Omega_{s,t}}f(x)-u_t(x,t) \, dx.
\end{align*}
So inequality \eqref{3-1} follows from the above equality and inequality \eqref{1-3.7}.

If  $v(x,t)$ is the solution to system \eqref{1-4}, $v(x,t)$ is  radial  and decreasing along the radial direction on $\Omega^\sharp$ for all $t$, hence $\Omega^\sharp_{s, t}$ is a round ball. Therefore both \eqref{1-3.4} and \eqref{1-3.5} hold as equalities  with $\th=1$ if we replace $u$ by $v$, hence equality \eqref{3-2} holds true.
\end{proof}
Now we turn to prove the main theorem.
\begin{proof}[Proof of Theorem \ref{thm}]
Recall
$$U(a,t)=\frac{1}{\theta }\int_{0}^{\theta a}u^*(s,t)ds$$
for $a\in[0, |\Omega^\sharp|)$, then the first derivative  of
 $a$ gives
$$
U'(a,t)=u^*(\th a, t),
$$
and from Lemma 1.1 of \cite{Ba76} it holds
\begin{align}\label{3-14}
U_t(a,t)=\frac{\p}{\p t}\Big( \frac{1}{\th}\int_{\Omega_{u^*,t}} u(x,t)\, dx\Big)= \frac{1}{\th}\int_{\Omega_{u^*,t}} u_t(x,t)\, dx,
\end{align}
where $u^*$ is evaluated at $(\th a, t)$ and $U'=\p_a U$.
Observing that
$$
\mu_u(u^*(\th a, t))=\th a,
$$
for a.e. $a>0$ (c.f. \cite[page 66]{MR86}), and differentiating above identity in $a$  yields
$$
\mu_u'(u^*(\th a, t))U''(a,t)=\th,
$$
then
\begin{align*}
 U''(a,t)=\frac{\th}{\mu_u'(u^*)}
    \ge -\frac{\int_{ \Omega_{u^*,t}}f(x)-u_t(x,t) \, dx}{\th\Phi^2(a) },
\end{align*}
where we used \eqref{3-1} in the inequality.
Recall from  \eqref{2-6} that
\begin{align}\label{3-15}
\int_{ \Omega_{u^*,t}}f(x)\, dx \le \int_0^{a\th} f^*(s)\,ds=\th \int_0^{a} (f^\sharp)^*(s)\,ds,
\end{align}
 then we deduce that
\begin{align}\label{3-16}
   U_t(a,t)-\Phi^2(a) U''(a,t)-\int_0^{a} (f^\sharp)^*(s)\,ds\le 0
\end{align}
for $(a,t)\in(0,|\Omega^\sharp|)\times(0, \infty)$ in the viscosity sense, where we used equality \eqref{3-14}.

Note that  $v(x,t)$ is radial function decreasing along the radial direction, and $f^\sharp$ is radial as well, then the inequality \eqref{3-15} holds as equality if we replace $u$ by $v$, $\Omega$ by $\Omega^\sharp$ and $f(x)$ by $f^\sharp$ according to equality case of Hardy-Littlewood inequality, i.e.
\begin{align}\label{3-17}
\int_{ \Omega^\sharp_{v^*,t}}f(x)\, dx =\int_0^{a} (f^\sharp)^*(s)\,ds.
\end{align}
Combining equality \eqref{3-2}  and  equality \eqref{3-17}, we find similarly as \eqref{3-16} that
\begin{align}\label{3-18}
   V_t(a,t)-\Phi^2(a) V''(a,t)-\int_0^{a} (f^\sharp)^*(s)\,ds= 0,
\end{align}
for $(a,t)\in(0, |\Omega^\sharp|)\times(0,\infty)$. It then follows from inequality \eqref{3-16} and equality \eqref{3-18} that
\begin{align}\label{3.13}
    (U-V)_t(a,t)-\Phi^2(a) (U-V)''(a, t)\le 0.
\end{align}
 On the other hand, by the initial condition of $u$ and $v$ we have
\begin{align}\label{3.14}
\begin{split}
    U(a, 0)-V(a,0)=& \frac{1}{\theta }\int_{0}^{\theta a}u^*(s,0)ds-\int_{0}^{a}v^*(s,0)ds\\
=&\frac{1}{\theta }\int_{0}^{\theta a}g^*(s)ds-\int_{0}^{a}(g^\sharp)^*(s)\, ds\\
=&0.
\end{split}
\end{align}
Moreover, by direct calculations we have
\begin{align}\label{3.15}
    U(0,t)=V(0,t)=0,
\end{align}
and
\begin{align}\label{3.16}
    U'(|\Omega^\sharp|,t)- V'(|\Omega^\sharp|,t)=u^*(|\Omega|,t)-v^*(|\Omega^\sharp|,t)=0.
\end{align}
We then conclude that differential inequality \eqref{3.13} holds with initial condition \eqref{3.14} and boundary conditions \eqref{3.15} and \eqref{3.16}.
Therefore \eqref{1-5} follows by the standard maximum principle for viscosity solutions to parabolic equation, see \cite[Theorem 3.3]{CIL92} and \cite[Section 2]{LW17}.

Now we consider the equality case of \eqref{1-5}.  If $U(a,t)=V(a,t)$, then  $ u^*(\theta a,t) = v^*(a ,t)$, hence $\mu_u(u^*,t)=\theta \mu_v(v^*,t)$
for all $a>0$. Then we estimate that
\begin{align}\label{3.17}
\begin{split}
|\p \Omega_{u^*,t}|^{2}&\leq-\mu_{u}'(u^*,t) \big(\int_{\Omega_{u^*,t}} f
 (x,t)-u_t(x,t)\ dx\big) \\
& \leq-\mu_{u}'(u^*,t) \big(\int_0^{\th a}f^*(s,t)\ ds-
\th U_t(a,t)\big)\\
&= -\theta^2\mu_{v}'(v^*,t) \big(\int_0^{a}(f^\sharp)^*(s,t)\ ds-
 V_t(a,t)\big)\\
&=\th^2|\p \Omega^\sharp_{v^*,t}|^{2}.
\end{split}
\end{align}
where in the first step we used inequality \eqref{3-1}, in the second inequality we used \eqref{3-14} and \eqref{3-15}.
On the other hand, the isoperimetric inequality \eqref{iso} gives
\begin{align*}
|\p \Omega_{u^*,t}|^{2}\ge \th^2|\p \Omega^\sharp_{v^*,t}|^{2},
\end{align*}
so inequality \eqref{3.17} holds as equality. Thus we conclude from the equality case in \eqref{iso} that $\Omega_{u^*,t}$ is a geodesic ball in $M_\k$, $\th=1$ and $u(x,t)=v(x,t)$ on $\Omega^\sharp \times [0,\infty)$, hence $f(x)=f^\sharp(x)$ and $g(x)=g^\sharp(x)$.
\end{proof}

\begin{proof}[Proof of Corollary \ref{Cor}] The proof of this lemma  is in the same spirit as in  Theorem 3.2 of \cite{Ba76}. For the reader's convenience, we give the details. For $p\ge 1$,  the Fubini's theorem gives
\begin{align*}
    \frac 1 \th \int_\Omega u^p(x,t)\, dx=\frac 1 \th \int_0^{|\Omega|} (u^*)^p(s, t) ds=\int_0^{|\Omega^\sharp|} (u^*)^p(\th s, t) \,ds
\end{align*}
and
\begin{align*}
   \int_{\Omega^\sharp} v^p(x,t)\, dx=\int_0^{|\Omega^\sharp|} (v^*)^p(s, t)\, ds.
\end{align*}
Then we estimate by integration by parts that
\begin{align}\label{3.18}
\begin{split}
  & \frac 1 \th \int_\Omega u^p(x,t)\, dx-\int_{\Omega^\sharp} v^p(x,t)\, dx\\
   =\quad &\int_0^{|\Omega^\sharp|} (u^*)^p(\th s, t)-(v^*)^p(s, t) \,ds\\
   \le \quad &\int_0^{|\Omega^\sharp|} p(u^*)^{p-1}(u^*-v^*)\,ds\\
   =\quad&(U-V)p(u^*)^{p-1}\Big|_0^{|\Omega^\sharp|}-\int_0^{|\Omega^\sharp|}(U-V)p(p-1)(u^*)^{p-2}\p_s u^*\,ds,
   \end{split}
\end{align}
where in the inequality we used the elementary inequality $
x^p-y^p\le px^{p-1}(x-y)
$
for $x, y\in \R$ when $p\ge 1$. By  \eqref{3.15}, \eqref{1-5} and  $\p_s u^*\le 0$, we see
\begin{align*}
    (U-V)p(u^*)^{p-1}\Big|_0^{|\Omega^\sharp|}=(U-V)p(u^*)^{p-1}\Big|_{s=|\Omega^\sharp|}\le 0,
\end{align*}
and
\begin{align*}
    (U-V)p(p-1)(u^*)^{p-2}\p_s u^*\ge 0,
\end{align*}
therefore it follows from \eqref{3.18} that
\begin{align*}
 \frac 1 \th \int_\Omega u^p(x,t)\, dx-\int_{\Omega^\sharp} v^p(x,t)\, dx\le 0
\end{align*}
for all $p\ge 1$, hence for $p=\infty$. We complete the proof of the corollary.
\end{proof}

In Euclidean space, the argument in the proof of \eqref{1-5}  can be adapted to prove  Bandle's comparison  for  the heat kernels.  This is outlined in Theorem 2.2 of \cite{Ba76}.  It’s not hard to adapt our argument  to prove Bandle's comparison  for  the heat kernels  on compact manifolds with positive Ricci curvature, and complete and noncompact manifolds with nonnegative Ricci curvature and positive asymptotic volume ratio. Hence one can easily obtain the Faber-Krahn inequality for the first Dirichlet eigenvalue on such manifolds via heat kernel comparison, which has  already been proved by Mattia-Lorenzo \cite{ML20} and Chen-Li \cite{CL21} respectively.

\end{document}